\documentclass[11pt,twoside]{article}
\usepackage[margin=1in]{geometry}
\usepackage{graphicx}
\usepackage{titlesec}
\usepackage{amsmath}
\usepackage{amsfonts}   
\usepackage{amssymb}    
\usepackage{amsthm}
\usepackage{mathtools}
\usepackage{verbatim}   
\usepackage{color}
\usepackage{listings}

\renewcommand{\(}{\left(}

\newtheorem{thm}{Theorem}

\newtheorem{lemma}{Lemma}
\newtheorem{cor}{Corollary}

\numberwithin{equation}{section}

\title{On the integral representations for   Dunkl  kernels of type $A_{2}$.}
\author{B\'{e}chir Amri}
\date{\small
Universit\'{e} Tunis El Manar\\
Facult\'{e} des Sciences de Tunis\\
D\'{e}partement de Math\'{e}matiques\\
Laboratoire d'Analyse Math\'{e}matiques et Applications LR11ES11\\
 2092- El Manar I, Tunis TUNISIA \\ bechir.amri@ipeit.rnu.tn }
\begin{document}
\maketitle
\begin{abstract}
We give an explicit integral formula for the Dunkl kernel associated to  root system of type $A_2$ and   parameter $k>0$, by  exploiting   recent result in \cite{BA1}.
\footnote{\par\textbf{Key words and phrases:}  Dunkl operators, root systems, reflection groups.
\par\textbf{2010 Mathematics Subject Classification:} Primary 33E30; Secondary 17B22,20F55.}

\end{abstract}
\section{Introduction}
 In this paper we mainly focus on  Dunkl kernels  associated to root systems of type $A$, for a purpose of finding  an explicit   representation  integrals for these  functions, following our recent work on symmetric case. We outline here a simple method  that leads us to  such formulas for the $A_2$ root system
 and provide a short and elementary proof of   Dunkl's formula for the intertwining operator  established in \cite{D1} for parameter $k>1/2$.
General references are \cite{D1,D2,D3,Je,R1,R2,R3}.

\par  Following the notations given in \cite{BA1},  letting $\mathbb{V}$ be the hyperplane,
$$\mathbb{V}=\{(x,y,z)\in \mathbb{R}^3;\; x+y+z=0\}$$ and the root system $R=\{\pm(e_1-e_2),\; \pm(e_1-e_3),\;\pm(e_2-e_3)\}$ where  $(e_1,eç_2,e_3)$ is the standard basis
of  the Euclidean space $\mathbb{R}^3$. Fixe $(e_1-e_2, e_2-e_3)$ as the basis of simple root and $C$ the corresponding  fundamental Weyl chamber,
$$C=\{\lambda=(\lambda_1,\lambda_2\lambda_3);\quad \lambda_3<\lambda_2<\lambda_1\}.$$
The Weyl group is isomorphic to the symmetric group  $S_3$. The  Dunkl operators are given by
$$T_i=\frac{\partial}{\partial x_i}+k\sum_{1\leq j\neq i\leq 3}\frac{1-s_{i,j}}{x_i-x_j},\qquad i=1,2,3$$
where $k$ is a positive  real parameter and $s_{i,j}$  acts on functions of vaiables $(x_1,x_2,x_3)$ by interchanging the variables $x_i$ and $x_j$.
The Dunkl kernel $E_k(.,y)$, $y\in \mathbb{R}^3$, characterized by being the unique  solution of the following eigenvalue problem
$$T_i(E_k(.,y))(x)=y_iE_k(x,y);\qquad E(0,y)=0,\quad i=1,2,3.$$
Let  $J_k$ the generalized Bessel function associated with $R$ and $k$,  given by
\begin{eqnarray}\label{k}
J_k(x,y)=\frac{1}{6}\sum_{\sigma\in G}E_k(\sigma.x,y).
\end{eqnarray}
The functions $J_k$ are related to the ordinary modified Bessel functions $\mathcal{J}_{k-\frac{1}{2}}$ by  (see \cite{BA1}):\\
\begin{eqnarray*}
J_{k}(\mu,\lambda)&=&\frac{\Gamma(3k)}{V(\lambda)^{2k-1}\Gamma(k)^3}\int_{\lambda_2}^{\lambda_1}
\int_{\lambda_3}^{\lambda_{2}}e^{\frac{(\mu_1+\mu_2-2\mu_3)(\nu_1+\nu_2)}{2}}
\mathcal{J}_{k-\frac{1}{2}}(\frac{(\mu_1-\mu_2)(\nu_1-\nu_2)}{2})\\&&\qquad \qquad\qquad\qquad\qquad\qquad(\nu_1-\nu_2)W_k(\mu,\lambda)d\nu_1d\nu_2,
\end{eqnarray*}
for all $\lambda=(\lambda_1,\lambda_2,\lambda_3)\in\mathbb{V}\cap C$  and $\mu\in \mathbb{R}^3$,
where
\begin{eqnarray*}
 V(\lambda)&= &(\lambda_1-\lambda_2)(\lambda_1-\lambda_3)(\lambda_1-\lambda_3)\\
W_k(\nu,\lambda)&=&\Big((\lambda_1-\nu_1)(\lambda_1-\nu_2)
(\lambda_2-\nu_2)(\nu_1-\lambda_2)(\nu_1-\lambda_3)(\nu_2-\lambda_3)\Big)^{k-1}.
\end{eqnarray*}
Recall here that
$$\mathcal{J}_{k-\frac{1}{2}}(z)= \frac{\Gamma(2k)}{2^{2k-1}\Gamma(k)^2}\int_{-1}^1e^{zt}(1-t^2)^{k-1}dt;\qquad z\in \mathbb{R}.$$

In the next section we shall use this fact to construct an  integral formula  for $E_k$. The following theorem is the main result of this article.
\begin{thm}\label{th1}
The Dunkl kernel of type $A_2$ has the following integral formula
\begin{eqnarray}\label{IF}
E_k(\mu,\lambda)&=& \frac{\Gamma(3k)}{V(\lambda)^{2k}\Gamma(k)^3}\int_{\lambda_2}^{\lambda_1}
\int_{\lambda_3}^{\lambda_{2}}
\bigg\{3(\lambda_1-\lambda_2)(\nu_1-\nu_2)\mathcal{J}_{k-\frac{1}{2}}\left(\frac{(\mu_1-\mu_2)(\nu_1-\nu_2)}{2}\right)\nonumber\\
&&-6\Big(\nu_1\nu_2+\frac{\lambda_3}{2}(\nu_1+\nu_2)+\lambda_1\lambda_2\Big)\mathcal{J}_{k-\frac{1}{2}}'\left(\frac{(\mu_1-\mu_2)(\nu_1-\nu_2)}{2}\right)\bigg\}
\nonumber\\&&
(\lambda_3-\nu_1)(\lambda_3-\nu_2)e^{\frac{(\mu_1+\mu_2-2\mu_3)(\nu_1+\nu_2)}{2}}W_{k}(\nu,\lambda)d\nu_1d\nu_2,
\end{eqnarray}
for all $\lambda\in\mathbb{V}\cap C$ and $\mu \in \mathbb{R}^3$.

\end{thm}
\section{Outline the proof}
An interesting relation between $J_k$ and $J_{k+1}$ is given in ( \cite{OP}, p.369 )  by the following functional equation
\begin{eqnarray}\label{p}
 T_V(J_{k+1}(.,y)V(.))(x)= \gamma_k J_k(x,y)
\end{eqnarray}
where $T_V=(T_1-T_2)(T_2-T_3(T_1-T_3))$ and $\gamma_k=T_V(V(.))(0)= \Big((2k+1)(3k+1)(3k+2)\Big)^{-1}$.
This together with  Proposition 1.4 of \cite{D3} implies
\begin{eqnarray}\label{k+1}
\sum_{\sigma\in G}det(\sigma)E_k(\sigma.\mu,\lambda)=\gamma_k V(\mu)V(\lambda)J_{k+1}(\mu,\lambda).
\end{eqnarray}
Combining  (\ref{k+1}) with (\ref{k}) yields for all $\mu\in \mathbb{R}^3$ and $\lambda\in \mathbb{V}$
\begin{eqnarray}
E_{k}(\mu,\lambda)+E_{k}(\mu,\sigma.\lambda)+E_{k}(\mu,\sigma^2.\lambda)=\frac{1}{2}\Big(\gamma_k V(\lambda)V(\mu)J_{k+1}(\mu,\lambda)+6J_{k}(\mu,\lambda)\Big)
\end{eqnarray}
where $\sigma=s_{1,3}s_{1,2}$. This is a  starting point  from which we have the following
\begin{lemma}\label{l1}
Let $\lambda\in \mathbb{V}$ and  $T$ be the operator
\begin{eqnarray*}
T=\frac{2\lambda_1+\lambda_2}{\lambda_1^2+\lambda_2^2+\lambda_1\lambda_2}T_1+\frac{2\lambda_2+\lambda_1}{\lambda_1^2+\lambda_2^2+\lambda_1\lambda_2}T_2+1
=\alpha(\lambda)T_1+\beta(\lambda)T_2+1
\end{eqnarray*}
Then  we have
\begin{eqnarray*}
E_{k}(\mu,\lambda)=
T\Big(\frac{\gamma_k}{6} \;V(\lambda)\; V(.)J_{k+1}(.,\lambda)+J_{k}(.,\lambda)\Big)(\mu), \qquad \mu \in \mathbb{R}^3.
\end{eqnarray*}
\end{lemma}
 The proof is a straightforward  calculation which we shall omit. However, to obtain our integral formula for $E_k$, it therefore comes down to express the following terms with suitable integrals
\begin{itemize}
  \item [(i)]   $ V(\mu)J_{k+1}(\mu,\lambda)$
\item [(ii)] $(\mu_1-\mu_2)(\mu_2-\mu_3) J_{k+1}(\mu,\lambda)$
\item [(iii)] $(\mu_1-\mu_2)(\mu_1-\mu_3) J_{k+1}(\mu,\lambda)$
\item [(iv)] $T_1(V(.)J_{k+1}(.,\lambda)(\mu)=V(\mu)\dfrac{\partial J_{k+1}}{\partial \mu_1}(\mu,\lambda)+(2k+1)\dfrac{\partial V(\mu)}{\partial \mu_1} J_{k+1}(\mu,\lambda)$
\item [(v)]  $T_2(V(.)J_{k+1}(.,\lambda)(\mu)=V(\mu)\dfrac{\partial J_{k+1}}{\partial \mu_2}(\mu,\lambda)+(2k+1)\dfrac{\partial V(\mu)}{\partial \mu_2} J_{k+1}(\mu,\lambda)$
\end{itemize}
    We will need to use the following
classical equations of the  modified Bessel function $\mathcal{J}_\alpha$, $\alpha> -\frac{1}{2}$,
\begin{eqnarray}\label{1}
    z\mathcal{J}_{\alpha+1}(z)&=&2(\alpha+1)\mathcal{J}_{\alpha}'(z)\\ \label{2}
\mathcal{J}_{\alpha}(z)&=& \mathcal{J}_{\alpha}''(z)+\frac{2\alpha+1}{z}\mathcal{J}_{\alpha}'(z)
\end{eqnarray}
and  the following facts:
\begin{eqnarray}
(\mu_1-\mu_2)(\mu_1-\mu_3)&=&\frac{(\mu_1-\mu_2)(\mu_1+\mu_2-2\mu_3)+(\mu_1-\mu_2)^2}{2}\label{3}
\\(\mu_1-\mu_2)(\mu_2-\mu_3)&=&\frac{(\mu_1-\mu_2)(\mu_1+\mu_2-2\mu_3)-(\mu_1-\mu_2)^2}{2}\label{4}
\\(\mu_1-\mu_3)(\mu_2-\mu_3)&=&\frac{(\mu_1+\mu_2-2\mu_3)^2-(\mu_1-\mu_2)^2}{4}\label{4'}
\\V(\mu)&=&\frac{(\mu_1+\mu_2-2\mu_3)^2(\mu_1-\mu_2)-(\mu_1-\mu_2)^3}{4}\;.\label{5}
\end{eqnarray}
First, from (\ref{1}) we have
\begin{eqnarray*}
&&(\mu_1-\mu_2)J_{k+1}(\mu,\lambda)\\&&\qquad\qquad\qquad=\frac{(4k+2)\Gamma(3k+3)}{V(\lambda)^{2k+1}\Gamma(k+1)^3}\int_{\lambda_2}^{\lambda_1}
\int_{\lambda_3}^{\lambda_{2}}e^{\frac{(\mu_1+\mu_2-2\mu_3)(\nu_1+\nu_2)}{2}}
\mathcal{J}_{k-\frac{1}{2}}'\left(\frac{(\mu_1-\mu_2)(\nu_1-\nu_2)}{2}\right)
\\&&\qquad\qquad\qquad\qquad\qquad\qquad\qquad\qquad\qquad\qquad W_{k+1}(\nu,\lambda)\;d\nu_1d\nu_2
\end{eqnarray*}
and by using integration by parts,
\begin{eqnarray*}
&&(\mu_1-\mu_2)^2J_{k+1}(\mu,\lambda)\\&&\qquad\qquad=\frac{(4k+2)\Gamma(3k+3)}{V(\lambda)^{2k+1}\Gamma(k+1)^3}\int_{\lambda_2}^{\lambda_1}
\int_{\lambda_3}^{\lambda_{2}}e^{\frac{(\mu_1+\mu_2-2\mu_3)(\nu_1+\nu_2)}{2}}
\mathcal{J}_{k-\frac{1}{2}}\left(\frac{(\mu_1-\mu_2)(\nu_1-\nu_2)}{2}\right)
\\&&\qquad\qquad\qquad\qquad\qquad\qquad\qquad\qquad\qquad(\partial_{\nu_1}-\partial_{\nu_2})W_{k+1}(\nu,\lambda)\;d\nu_1d\nu_2.
\end{eqnarray*}
Making use of  (\ref{2}) we have
\\\\
$\displaystyle{(\mu_1-\mu_2)^3J_{k+1}(\mu,\lambda)}$
\begin{eqnarray*}
&=&-\;\frac{(4k+2)\Gamma(3k+3)}{V(\lambda)^{2k+1}\Gamma(k+1)^3}\int_{\lambda_2}^{\lambda_1}
\int_{\lambda_3}^{\lambda_{2}}(\mu_1-\mu_2)e^{\frac{(\mu_1+\mu_2-2\mu_3)(\nu_1+\nu_2)}{2}}
\mathcal{J}_{k-\frac{1}{2}}''\left(\frac{(\mu_1-\mu_2)(\nu_1-\nu_2)}{2}\right)
\\&&\qquad\qquad\qquad\qquad\qquad\qquad\qquad (\partial_{\nu_1}-\partial_{\nu_2})W_{k+1}(\nu,\lambda)\;d\nu_1d\nu_2\\
&&-\;\frac{4k(4k+2)\Gamma(3k+3)}{V(\lambda)^{2k+1}\Gamma(k+1)^3}\int_{\lambda_2}^{\lambda_1}
\int_{\lambda_3}^{\lambda_{2}}e^{\frac{(\mu_1+\mu_2-2\mu_3)(\nu_1+\nu_2)}{2}}
\mathcal{J}_{k-\frac{1}{2}}'\left(\frac{(\mu_1-\mu_2)(\nu_1-\nu_2)}{2}\right)
\\&&\qquad\qquad\qquad\qquad\qquad\qquad\qquad\frac{(\partial_{\nu_1}-\partial_{\nu_2})W_{k+1}(\nu,\lambda)}{\nu_1-\nu_2}\;d\nu_1d\nu_2.
\end{eqnarray*}
and by integration by parts one-time,
\\\\
$\displaystyle{(\mu_1+\mu_2-2\mu_3)^2(\mu_1-\mu_2)J_{k+1}(\mu,\lambda)}$
\begin{eqnarray*}
&=-&\frac{(4k+2)\Gamma(3k+3)}{V(\lambda)^{2k+1}\Gamma(k+1)^3}\int_{\lambda_2}^{\lambda_1}
\int_{\lambda_3}^{\lambda_{2}}(\mu_1+\mu_2-2\mu_3)e^{\frac{(\mu_1+\mu_2-2\mu_3)(\nu_1+\nu_2)}{2}}
\mathcal{J}_{k-\frac{1}{2}}'\left(\frac{(\mu_1-\mu_2)(\nu_1-\nu_2)}{2}\right)
\\&&\qquad\qquad\qquad\qquad\qquad\qquad(\partial_{\nu_1}+\partial_{\nu_2})W_{k+1}(\nu,\lambda)d\nu_1d\nu_2.
\end{eqnarray*}
 Note that the condition $k>0$ is not sufficient to make
an integration by parts again using the derivative operators $\partial_{\nu_1}+\partial_{\nu_2}$ or $\partial_{\nu_1}-\partial_{\nu_2}$, because the appearance
of  $\partial_{\nu_1}^2W_{k+1}$ and $\partial_{\nu_2}^2W_{k+1}$. However, we see that
\begin{eqnarray*}
&-&(\mu_1+\mu_2-2\mu_3)e^{\frac{(\mu_1+\mu_2-2\mu_3)(\nu_1+\nu_2)}{2}}
\mathcal{J}_{k-\frac{1}{2}}'\left(\frac{(\mu_1-\mu_2)(\nu_1-\nu_2)}{2}\right)(\partial_{\nu_1}+\partial_{\nu_2})W_{k+1}(\nu,\lambda)
\\&+&  (\mu_1-\mu_2)e^{\frac{(\mu_1+\mu_2-2\mu_3)(\nu_1+\nu_2)}{2}}
\mathcal{J}_{k-\frac{1}{2}}''\left(\frac{(\mu_1-\mu_2)(\nu_1-\nu_2)}{2}\right)(\partial_{\nu_1}-\partial_{\nu_2})W_{k+1}(\nu,\lambda)
\\&=&-2\partial_{\nu_1}\left\{e^{\frac{(\mu_1+\mu_2-2\mu_3)(\nu_1+\nu_2)}{2}}
\mathcal{J}_{k-\frac{1}{2}}'\left(\frac{(\mu_1-\mu_2)(\nu_1-\nu_2)}{2}\right)\right\}\partial_{\nu_2}W_{k+1}(\nu,\lambda)
\\&&-2\partial_{\nu_2}\left\{e^{\frac{(\mu_1+\mu_2-2\mu_3)(\nu_1+\nu_2)}{2}}
\mathcal{J}_{k-\frac{1}{2}}'\left(\frac{(\mu_1-\mu_2)(\nu_1-\nu_2)}{2}\right)\right\}\partial_{\nu_1}W_{k+1}(\nu,\lambda).
\end{eqnarray*}
Thus from  (\ref{5}) and integration by parts we obtain
\\\\
$\displaystyle{V(\mu)J_{k+1}(\mu,\lambda)}$
\begin{eqnarray*}
&=&\frac{(4k+2)\Gamma(3k+3)}{V(\lambda)^{2k+1}\Gamma(k+1)^3}\int_{\lambda_2}^{\lambda_1}
\int_{\lambda_3}^{\lambda_{2}}e^{\frac{(\mu_1+\mu_2-2\mu_3)(\nu_1+\nu_2)}{2}}
\mathcal{J}_{k-\frac{1}{2}}'\left(\frac{(\mu_1-\mu_2)(\nu_1-\nu_2)}{2}\right)
\\&&\qquad\qquad\qquad\qquad\qquad\qquad\left( \partial_{\nu_1}\partial_{\nu_2}+
k\frac{\partial_{\nu_1}-\partial_{\nu_2}}{\nu_1-\nu_2}\right)W_{k+1}(\nu,\lambda)d\nu_1d\nu_2
\end{eqnarray*}
which is a nice integral formula for $(i)$.\par
Next, using (\ref{3})and (\ref{4}) with integration by parts,
\begin{eqnarray*}
&&(\mu_1-\mu_2)(\mu_1-\mu_3)J_{k+1}(\mu,\lambda)\\&&\qquad=
-\; \frac{(4k+2)\Gamma(3k+3)}{V(\lambda)^{2k+1}\Gamma(k+1)^3}\int_{\lambda_2}^{\lambda_1}
\int_{\lambda_3}^{\lambda_{2}}e^{\frac{(\mu_1+\mu_2-2\mu_3)(\nu_1+\nu_2)}{2}}
\mathcal{J}_{k-\frac{1}{2}}\left(\frac{(\mu_1-\mu_2)(\nu_1-\nu_2)}{2}\right)
\\&&\qquad\qquad\qquad\qquad\qquad\qquad\qquad\qquad\quad(\partial_{\nu_1}-\partial_{\nu_2})W_{k+1}(\nu,\lambda)d\nu_1d\nu_2\\
&&\qquad\qquad-\frac{(4k+2)\Gamma(3k+3)}{V(\lambda)^{2k+1}\Gamma(k+1)^3}\int_{\lambda_2}^{\lambda_1}
\int_{\lambda_3}^{\lambda_{2}}e^{\frac{(\mu_1+\mu_2-2\mu_3)(\nu_1+\nu_2)}{2}}
\mathcal{J}_{k-\frac{1}{2}}'\left(\frac{(\mu_1-\mu_2)(\nu_1-\nu_2)}{2}\right)
\\&&\qquad\qquad\qquad\qquad\qquad\qquad\qquad\qquad\quad(\partial_{\nu_1}+\partial_{\nu_2}) W_{k+1}(\nu,\lambda)d\nu_1d\nu_2
\end{eqnarray*}
and
\begin{eqnarray*}
&&(\mu_1-\mu_2)(\mu_2-\mu_3)J_{k+1}(\mu,\lambda)\\&&\qquad\qquad=
\frac{(2k+1)\Gamma(3k+3)}{V(\lambda)^{2k+1}\Gamma(k+1)^3}\int_{\lambda_2}^{\lambda_1}
\int_{\lambda_3}^{\lambda_{2}}e^{\frac{(\mu_1+\mu_2-2\mu_3)(\nu_1+\nu_2)}{2}}
\mathcal{J}_{k-\frac{1}{2}}\left(\frac{(\mu_1-\mu_2)(\nu_1-\nu_2)}{2}\right)
\\&&\qquad\qquad\qquad\qquad\qquad\qquad\qquad\qquad\qquad(\partial_{\nu_1}-\partial_{\nu_2})W_{k+1}(\lambda,\mu)d\nu_1d\nu_2\\
&&-\frac{(2k+1)\Gamma(3k+3)}{V(\lambda)^{2k+1}\Gamma(k+1)^3}\int_{\lambda_2}^{\lambda_1}
\int_{\lambda_3}^{\lambda_{2}}e^{\frac{(\mu_1+\mu_2-2\mu_3)(\nu_1+\nu_2)}{2}}
\mathcal{J}_{k-\frac{1}{2}}'(\frac{(\mu_1-\mu_2)(\nu_1-\nu_2)}{2})
\\&&\qquad\qquad\qquad\qquad\qquad\qquad\qquad\qquad\qquad\qquad(\partial_{\nu_1}+\partial_{\nu_2}) W_{k+1}(\nu,\lambda)d\nu_1d\nu_2.
\end{eqnarray*}
For (iv) we  make use of the fact that
$$z\mathcal{J}'_{\alpha+1}(z)=2(\alpha+1)\Big(\mathcal{J}_\alpha(z)-\mathcal{J}_{\alpha+1}(z)\Big),$$
and write
\begin{eqnarray*}
 &&V(\mu)\frac{\partial J_{k+1}}{\partial \mu_1}(\mu)\\&&=\frac{\Gamma(3k+3)}{2V(\lambda)^{2k+1}\Gamma(k+1)^3}
 V(\mu)\int_{\lambda_2}^{\lambda_1}
\int_{\lambda_3}^{\lambda_{2}}e^{\frac{(\mu_1+\mu_2-2\mu_3)(\nu_1+\nu_2)}{2}}
\mathcal{J}_{k+\frac{1}{2}}\left(\frac{(\mu_1-\mu_2)(\nu_1-\nu_2)}{2}\right)\\&&\qquad \qquad\qquad\qquad\qquad\qquad\qquad\qquad\qquad(\nu_1-\nu_2)(\nu_1+\nu_2)W_{k+1}(\nu,\lambda)d\nu_1d\nu_2
\\&&+\frac{(2k+1)\Gamma(3k+3)}{V(\lambda)^{2k+1}\Gamma(k+1)^3}(\mu_1-\mu_3)(\mu_2-\mu_3)\int_{\lambda_2}^{\lambda_1}
\int_{\lambda_3}^{\lambda_{2}}e^{\frac{(\mu_1+\mu_2-2\mu_3)(\nu_1+\nu_2)}{2}}
\mathcal{J}_{k-\frac{1}{2}}\left(\frac{(\mu_1-\mu_2)(\nu_1-\nu_2)}{2}\right)\\&&\qquad\qquad\qquad\qquad\qquad\qquad\qquad\qquad\qquad\qquad\qquad
(\nu_1-\nu_2)W_{k+1}(\nu,\lambda)d\nu_1d\nu_2\\
&&-(2k+1)(\mu_1-\mu_3)(\mu_2-\mu_3)J_{k+1}\;.
\end{eqnarray*}
Proceeding as for the integral representation of $(i)$, we have
\begin{eqnarray*}
&&\frac{\Gamma(3k+3)}{2V(\lambda)^{2k+1}\Gamma(k+1)^3}
 \Big\{V(\mu)\int_{\lambda_2}^{\lambda_1}
\int_{\lambda_3}^{\lambda_{2}}e^{\frac{(\mu_1+\mu_2-2\mu_3)(\nu_1+\nu_2)}{2}}
\mathcal{J}_{k+\frac{1}{2}}\left(\frac{(\mu_1-\mu_2)(\nu_1-\nu_2)}{2}\right)\\&&\qquad \qquad\qquad\qquad\qquad\qquad(\nu_1-\nu_2)(\nu_1+\nu_2)W_{k+1}(\nu,\lambda)d\nu_1d\nu_2\Big\}
\\&=&\frac{(2k+1)\Gamma(3k+3)}{V(\lambda)^{2k+1}\Gamma(k+1)^3}\int_{\lambda_2}^{\lambda_1}
\int_{\lambda_3}^{\lambda_{2}}e^{\frac{(\mu_1+\mu_2-2\mu_3)(\nu_1+\nu_2)}{2}}
\mathcal{J}_{k-\frac{1}{2}}'\left(\frac{(\mu_1-\mu_2)(\nu_1-\nu_2)}{2}\right)
\\&&\qquad\qquad\left\{\partial_{\nu_1}\partial_{\nu_2}\Big((\nu_1+\nu_2)W_{k+1}(\nu,\lambda)\Big)
+k\;\frac{(\partial_{\nu_1}-\partial_{\nu_2})\Big((\nu_1+\nu_2)W_{k+1}(\nu,\lambda)\Big)}{\nu_1-\nu_2}\right\}d\nu_1d\nu_2\;.
\end{eqnarray*}
On the other hand, by using (\ref{2}) and (\ref{4'}) with integration by parts,
\begin{eqnarray*}
&&\frac{(2k+1)\Gamma(3k+3)}{V(\lambda)^{2k+1}\Gamma(k+1)^3}(\mu_1-\mu_3)(\mu_2-\mu_3)\int_{\lambda_2}^{\lambda_1}
\int_{\lambda_3}^{\lambda_{2}}e^{\frac{(\mu_1+\mu_2-2\mu_3)(\nu_1+\nu_2)}{2}}
\mathcal{J}_{k-\frac{1}{2}}\left(\frac{(\mu_1-\mu_2)(\nu_1-\nu_2)}{2}\right)\\ &&\qquad\qquad\qquad\qquad\qquad \qquad\qquad\qquad\qquad\qquad\qquad(\nu_1-\nu_2)W_{k+1}(\nu,\lambda)d\nu_1d\nu_2
\\&=& -\frac{(2k+1)\Gamma(3k+3)}{4V(\lambda)^{2k+1}\Gamma(k+1)^3}\int_{\lambda_2}^{\lambda_1}
\int_{\lambda_3}^{\lambda_{2}}(\mu_1+\mu_2-2\mu_3)e^{\frac{(\mu_1+\mu_2-2\mu_3)(\nu_1+\nu_2)}{2}}
\mathcal{J}_{k-\frac{1}{2}}\left(\frac{(\mu_1-\mu_2)(\nu_1-\nu_2)}{2}\right)\\ &&\qquad\qquad\qquad\qquad\qquad \qquad\qquad(\partial\nu_1+\partial\nu_2)\Big(\nu_1-\nu_2\Big)W_{k+1}(\nu,\lambda)d\nu_1d\nu_2
\\&+& \frac{(2k+1)\Gamma(3k+3)}{4V(\lambda)^{2k+1}\Gamma(k+1)^3}\int_{\lambda_2}^{\lambda_1}
\int_{\lambda_3}^{\lambda_{2}}(\mu_1-\mu_2)e^{\frac{(\mu_1+\mu_2-2\mu_3)(\nu_1+\nu_2)}{2}}
\mathcal{J}_{k-\frac{1}{2}}'\left(\frac{(\mu_1-\mu_2)(\nu_1-\nu_2)}{2}\right)
\\ &&\qquad\qquad\qquad\qquad\qquad \qquad(\partial\nu_1-\partial\nu_2)\Big(\nu_1-\nu_2)W_{k+1}(\nu,\lambda)\Big)d\nu_1d\nu_2
\\&+& \frac{k(2k+1)\Gamma(3k+3)}{V(\lambda)^{2k+1}\Gamma(k+1)^3}\int_{\lambda_2}^{\lambda_1}
\int_{\lambda_3}^{\lambda_{2}}e^{\frac{(\mu_1+\mu_2-2\mu_3)(\nu_1+\nu_2)}{2}}
\mathcal{J}_{k-\frac{1}{2}}\left(\frac{(\mu_1-\mu_2)(\nu_1-\nu_2)}{2}\right)
\\ &&\qquad\qquad\qquad\qquad\qquad \qquad(\partial\nu_1-\partial\nu_2)W_{k+1}(\nu,\lambda)d\nu_1d\nu_2.
\end{eqnarray*}
As we noted above for the use  of integration  by parts a second time, we can   do it by  the  following observations
\begin{eqnarray*}
&-&(\mu_1+\mu_2-2\mu_3)e^{\frac{(\mu_1+\mu_2-2\mu_3)(\nu_1+\nu_2)}{2}}
\mathcal{J}_{k-\frac{1}{2}}\left(\frac{(\mu_1-\mu_2)(\nu_1-\nu_2)}{2}\right)(\partial_{\nu_1}+\partial_{\nu_2})\Big((\nu_1-\nu_2)W_{k+1}(\nu,\lambda)\Big)
\\&+&  (\mu_1-\mu_2)e^{\frac{(\mu_1+\mu_2-2\mu_3)(\nu_1+\nu_2)}{2}}
\mathcal{J}_{k-\frac{1}{2}}'\left(\frac{(\mu_1-\mu_2)(\nu_1-\nu_2)}{2}\right)(\partial_{\nu_1}-\partial_{\nu_2})\Big((\nu_1-\nu_2)W_{k+1}(\nu,\lambda)\Big)
\\&=&-2\partial_{\nu_1}\left\{e^{\frac{(\mu_1+\mu_2-2\mu_3)(\nu_1+\nu_2)}{2}}
\mathcal{J}_{k-\frac{1}{2}}'\left(\frac{(\mu_1-\mu_2)(\nu_1-\nu_2)}{2}\right)\right\}\partial_{\nu_2}\Big((\nu_1-\nu_2)W_{k+1}(\nu,\lambda)\Big)
\\&&-2\partial_{\nu_2}\left\{e^{\frac{(\mu_1+\mu_2-2\mu_3)(\nu_1+\nu_2)}{2}}
\mathcal{J}_{k-\frac{1}{2}}'\left(\frac{(\mu_1-\mu_2)(\nu_1-\nu_2)}{2}\right)\right\}\partial_{\nu_1}\Big((\nu_1-\nu_2)W_{k+1}(\nu,\lambda)\Big).
\end{eqnarray*}
Thus
 \begin{eqnarray*}
&&\frac{(2k+1)\Gamma(3k+3)}{V(\lambda)^{2k+1}\Gamma(k+1)^3}(\mu_1-\mu_3)(\mu_2-\mu_3)\int_{\lambda_2}^{\lambda_1}
\int_{\lambda_3}^{\lambda_{2}}e^{\frac{(\mu_1+\mu_2-2\mu_3)(\nu_1+\nu_2)}{2}}
\mathcal{J}_{k-\frac{1}{2}}\left(\frac{(\mu_1-\mu_2)(\nu_1-\nu_2)}{2}\right)\\ &&\qquad\qquad\qquad\qquad\qquad \qquad\qquad\qquad\qquad\qquad\qquad(\nu_1-\nu_2)W_{k+1}(\nu,\lambda)d\nu_1d\nu_2
\\&=&
\frac{(2k+1)\Gamma(3k+3)}{V(\lambda)^{2k+1}\Gamma(k+1)^3}\int_{\lambda_2}^{\lambda_1}
\int_{\lambda_3}^{\lambda_{2}}e^{\frac{(\mu_1+\mu_2-2\mu_3)(\nu_1+\nu_2)}{2}}
\mathcal{J}_{k-\frac{1}{2}}(\frac{(\mu_1-\mu_2)(\nu_1-\nu_2)}{2})
\\&&\qquad\qquad\qquad\qquad\qquad\left\{\partial\nu_1\partial\nu_2\Big((\nu_1-\nu_2)W_{k+1}(\nu,\lambda)\Big)+k (\partial\nu_1-\partial\nu_2)W_{k+1}(\nu,\lambda)\right\}d\nu_1d\nu_2\;.
\end{eqnarray*}
From these calculations it follows that
\begin{eqnarray*}
&&T_1(V (.)J_{k+1}(.,\lambda))(\mu)\\
&&=V(\mu)\frac{\partial J_{k+1}}{\partial \mu_1}(\mu)+(2k+1)\Big((\mu_1-\mu_3)(\mu_2-\mu_3)+(\mu_1-\mu_2)(\mu_2-\mu_3)\Big) J_{k+1}(\mu)\qquad\qquad\\&&=
\frac{(2k+1)\Gamma(3k+3)}{V(\lambda)^{2k+1}\Gamma(k+1)^3}\int_{\lambda_2}^{\lambda_1}
\int_{\lambda_3}^{\lambda_{2}}e^{\frac{(\mu_1+\mu_2-2\mu_3)(\nu_1+\nu_2)}{2}}
\mathcal{J}_{k-\frac{1}{2}}'\left(\frac{(\mu_1-\mu_2)(\nu_1-\nu_2)}{2}\right)
\\&& \qquad\qquad\qquad\left\{(\nu_1+\nu_2)\left(\partial\nu_1\partial\nu_2
 +k\frac{\partial_{\nu_1}-\partial_{\nu_2}}{\nu_1-\nu_2} \right)
 -2k(\partial_{\nu_1}+\partial_{\nu_2})\right\}W_{k+1}(\nu,\lambda)d\nu_1d\nu_2
\\&+&
\frac{(2k+1)\Gamma(3k+3)}{V(\lambda)^{2k+1}\Gamma(k+1)^3}\int_{\lambda_2}^{\lambda_1}
\int_{\lambda_3}^{\lambda_{2}}e^{\frac{(\mu_1+\mu_2-2\mu_3)(\nu_1+\nu_2)}{2}}
\mathcal{J}_{k-\frac{1}{2}}\left(\frac{(\mu_1-\mu_2)(\nu_1-\nu_2)}{2}\right)
\\&&\qquad\qquad\qquad\qquad\qquad\qquad\qquad(\nu_1-\nu_2)\left(\partial\nu_1\partial\nu_2
 +3k\frac{(\partial_{\nu_1}-\partial_{\nu_2})}{\nu_1-\nu_2}\right)W_{k+1}(\nu,\lambda)d\nu_1d\nu_2
\end{eqnarray*}

By the fact that
$$T_2(V (.)J_{k+1}(.,\lambda))(\mu_1,\mu_2,\mu_3)=-T_1(V (.)J_{k+1}(.,\lambda))(\mu_2,\mu_1,\mu_3)$$
we also have
\begin{eqnarray*}
&&T_2(V (.)J_{k+1}(.,\lambda))(\mu)
\\&&=
\frac{(2k+1)\Gamma(3k+3)}{V(\lambda)^{2k+1}\Gamma(k+1)^3}\int_{\lambda_2}^{\lambda_1}
\int_{\lambda_3}^{\lambda_{2}}e^{\frac{(\mu_1+\mu_2-2\mu_3)(\nu_1+\nu_2)}{2}}
\mathcal{J}_{k-\frac{1}{2}}'\left(\frac{(\mu_1-\mu_2)(\nu_1-\nu_2)}{2}\right)
\\&& \qquad\qquad\left\{(\nu_1+\nu_2)\left(\partial\nu_1\partial\nu_2
 +k\frac{\partial_{\nu_1}-\partial_{\nu_2}}{\nu_1-\nu_2} \right)
 -2k(\partial_{\nu_1}+\partial_{\nu_2})\right\}W_{k+1}(\nu,\lambda)d\nu_1d\nu_2
\\&-&
\frac{(2k+1)\Gamma(3k+3)}{V(\lambda)^{2k+1}\Gamma(k+1)^3}\int_{\lambda_2}^{\lambda_1}
\int_{\lambda_3}^{\lambda_{2}}e^{\frac{(\mu_1+\mu_2-2\mu_3)(\nu_1+\nu_2)}{2}}
\mathcal{J}_{k-\frac{1}{2}}\left(\frac{(\mu_1-\mu_2)(\nu_1-\nu_2)}{2}\right)(\nu_1-\nu_2)
\\&&\qquad\qquad\qquad\qquad\qquad\qquad\left(\partial\nu_1\partial\nu_2
 +3k\frac{(\partial_{\nu_1}-\partial_{\nu_2})}{\nu_1-\nu_2}\right)W_{k+1}(\nu,\lambda)d\nu_1d\nu_2\;.
\end{eqnarray*}
\par By virtue of these integral  formulas we obtain
\begin{eqnarray*}
&&T(V (.)J_{k+1}(.,\lambda))(\mu)
\\&&=
\frac{(2k+1)\Gamma(3k+3)}{V(\lambda)^{2k+1}\Gamma(k+1)^3}\int_{\lambda_2}^{\lambda_1}
\int_{\lambda_3}^{\lambda_{2}}e^{\frac{(\mu_1+\mu_2-2\mu_3)(\nu_1+\nu_2)}{2}}
\mathcal{J}_{k-\frac{1}{2}}'\left(\frac{(\mu_1-\mu_2)(\nu_1-\nu_2)}{2}\right)
\\&& \left\{((\alpha+\beta)(\nu_1+\nu_2)+2)\left(\partial\nu_1\partial\nu_2
 +k\frac{\partial_{\nu_1}-\partial_{\nu_2}}{\nu_1-\nu_2} \right)
 -2k(\alpha+\beta)(\partial_{\nu_1}+\partial_{\nu_2})\right\}W_{k+1}(\nu,\lambda)d\nu_1d\nu_2
\\&+&
\frac{(2k+1)\Gamma(3k+3)}{V(\lambda)^{2k+1}\Gamma(k+1)^3}\int_{\lambda_2}^{\lambda_1}
\int_{\lambda_3}^{\lambda_{2}}e^{\frac{(\mu_1+\mu_2-2\mu_3)(\nu_1+\nu_2)}{2}}
\mathcal{J}_{k-\frac{1}{2}}\left(\frac{(\mu_1-\mu_2)(\nu_1-\nu_2)}{2}\right)(\nu_1-\nu_2)
\\&&\qquad\qquad\qquad\qquad\qquad\qquad (\alpha-\beta)\left(\partial\nu_1\partial\nu_2
 +3k\frac{\partial_{\nu_1}-\partial_{\nu_2}}{\nu_1-\nu_2}\right)W_{k+1}(\nu,\lambda)d\nu_1d\nu_2\;.
\end{eqnarray*}
Put $a(\lambda)=\lambda_1\lambda_2+\lambda_1\lambda_3+\lambda_2\lambda_3$ and
 $b(\lambda)=-\lambda_1\lambda_2\lambda_3$,  we have

\begin{eqnarray*}
&&\left(\partial\nu_1\partial\nu_2
 +k\frac{\partial_{\nu_1}-\partial_{\nu_2}}{\nu_1-\nu_2} \right)W_{k+1}(\nu,\lambda)
 =-k^2\Big(6\nu_1^2\nu_2^2+2a(\nu_1^2+\nu_2^2+\nu_1\nu_2)+3b(\nu_1+\nu_2)\Big)W_{k}(\nu,\lambda)\\
 &&\left\{(\nu_1+\nu_2)\left(\partial\nu_1\partial\nu_2
 +k\frac{\partial_{\nu_1}-\partial_{\nu_2}}{\nu_1-\nu_2} \right)
 -2k(\partial_{\nu_1}+\partial_{\nu_2})\right\}W_{k+1}(\nu,\lambda)\\&&\qquad\qquad\qquad\qquad\qquad\qquad=k^2\Big( 2a\nu_1\nu_2(\nu_1+\nu_2)+3b(\nu_1-\nu_2)^2+2a^2(\nu_1+\nu_2)+4ab\Big)W_{k}(\nu,\lambda)
 \\ &&\left(\partial\nu_1\partial\nu_2 +3k\frac{(\partial_{\nu_1}-\partial_{\nu_2})}{\nu_1-\nu_2}\right)W_{k+1}(\nu,\lambda)
 =k^2\Big(-6a\nu_1\nu_2-9b(\nu_1+\nu_2)+2a^2\Big)W_{k}(\nu,\lambda)
  \\&&\left\{((\alpha+\beta)(\nu_1+\nu_2)+2)\left(\partial\nu_1\partial\nu_2
 +k\frac{\partial_{\nu_1}-\partial_{\nu_2}}{\nu_1-\nu_2} \right)
 -2k(\alpha+\beta)(\partial_{\nu_1}+\partial_{\nu_2})\right\}W_{k+1}(\nu,\lambda)\\&&
 =-k^2\Big(12\nu_1^2\nu_2^2+4a(\nu_1^2+\nu_2^2+\nu_1\nu_2)+6b(\nu_1+\nu_2)\Big)W_{k}(\nu,\lambda)\\&&
 \qquad\qquad\qquad\qquad\qquad+(\alpha+\beta)k^2\Big( 2a\nu_1\nu_2(\nu_1+\nu_2)+3b(\nu_1-\nu_2)^2+2a^2(\nu_1+\nu_2)+4ab\Big)W_{k}(\nu,\lambda).
 \end{eqnarray*}
We finally obtain
\begin{eqnarray*}
E_k(\mu,\lambda)&=& \frac{\Gamma(3k)}{V(\lambda)^{2k}\Gamma(k)^3}\int_{\lambda_2}^{\lambda_1}
\int_{\lambda_3}^{\lambda_{2}}e^{\frac{(\mu_1+\mu_2-2\mu_3)(\nu_1+\nu_2)}{2}}\mathcal{J}_{k-\frac{1}{2}}\left(\frac{(\mu_1-\mu_2)(\nu_1-\nu_2)}{2}\right)
(\nu_1-\nu_2)\\&&\left\{\frac{\alpha-\beta}{2}\Big(-6a\nu_1\nu_2-9b(\nu_1+\nu_2)+2a^2\Big)+\Big(\frac{\alpha+\beta}{2}(\nu_1+\nu_2)+1\Big)V(\lambda)\right\}
W_{k}(\nu,\lambda)d\nu_1d\nu_2
\\&+&\frac{\Gamma(3k)}{V(\lambda)^{2k}\Gamma(k)^3}\int_{\lambda_2}^{\lambda_1}
\int_{\lambda_3}^{\lambda_{2}}e^{\frac{(\mu_1+\mu_2-2\mu_3)(\nu_1+\nu_2)}{2}}\mathcal{J}_{k-\frac{1}{2}}'(\frac{(\mu_1-\mu_2)(\nu_1-\nu_2)}{2})
\\&&\Big\{\frac{(\alpha+\beta)}{2}\Big( 2a\nu_1\nu_2(\nu_1+\nu_2)+3b(\nu_1-\nu_2)^2+2a^2(\nu_1+\nu_2)+4ab\Big)
\\&-&\Big(6\nu_1^2\nu_2^2+2a(\nu_1^2+\nu_2^2+\nu_1\nu_2)+3b(\nu_1+\nu_2)\Big)+\frac{\alpha-\beta}{2}(\nu_1-\nu_2)^2V(\lambda)\Big\}
 W_{k}(\nu,\lambda)d\nu_1d\nu_2
\end{eqnarray*}
where,
\begin{eqnarray*}
&&\frac{\alpha-\beta}{2}\Big(-6a\nu_1\nu_2-9b(\nu_1+\nu_2)+2a^2\Big)+\Big(\frac{\alpha+\beta}{2}(\nu_1+\nu_2)+1\Big)V(\lambda)
\\&=&3(\lambda_1-\lambda_2)\nu_1\nu_2+3(\lambda_1^2-\lambda_2^2)(\nu_1+\nu_2)+3\lambda_3^2(\lambda_1-\lambda_2)
\\&=&3(\lambda_1-\lambda_2)(\lambda_3-\nu_1)(\lambda_3-\nu_2),
\end{eqnarray*}
\begin{eqnarray*}
&&\Big\{\frac{(\alpha+\beta)}{2}\Big( 2a\nu_1\nu_2(\nu_1+\nu_2)+3b(\nu_1-\nu_2)^2+2a^2(\nu_1+\nu_2)+4ab\Big)
\\&-&\Big(6\nu_1^2\nu_2^2+2a(\nu_1^2+\nu_2^2+\nu_1\nu_2)+3b(\nu_1+\nu_2)\Big)+\frac{\alpha-\beta}{2}(\nu_1-\nu_2)^2V(\lambda)\Big\}
\\&=&-6\nu_1^2\nu_2^2+3\lambda_3\nu_1\nu_2(\nu_1+\nu_2)-3\lambda_3(\lambda_1^2+\lambda_2^2)(\nu_1+\nu_2)-6\lambda_1\lambda_2\lambda_3^2-
2(\nu_1^2+\nu_2^2+\nu_1\nu_2)(\lambda_1\lambda_2-\lambda_3^2)\\
&&+(2\lambda_1\lambda_2+\lambda_3^2)(\nu_1-\nu_2)^2\\
&=&-6(\lambda_3-\nu_1)(\lambda_3-\nu_2)\Big(\nu_1\nu_2+\frac{\lambda_3}{2}(\nu_1+\nu_2)+\lambda_1\lambda_2\Big).\\
 \end{eqnarray*}
This conclude the proof  of Theorem the main result.
\par Now if we  equippped the space $\mathbb{V}$ with the basis $(\; e_1-e_3,\; e_2 -e_3\;)$ and with  the Lebesgue measure $d\nu=d\nu_1d\nu_2$, then we can state
\begin{cor} The Dunkl kernel $E_k$  connected with the exponential function by
\begin{eqnarray}\label{F1}
E_{k}(\mu,\lambda)= \int_{co(\lambda) }
 e^{\langle\mu,\nu \rangle} F_k\left(\frac{\nu_1+\nu_2}{2},\frac{\nu_1-\nu_2}{2},\lambda\right)\;d\nu
\end{eqnarray}
where $co(\lambda)=\{\nu\in \mathbb{V},\; \lambda_3\leq \;\nu_1,\;\nu_2,\;\nu_3\;\leq \lambda_1\}$, the convex hull of the orbit $G.\lambda$ and the  function
$F_k$ is given by
\begin{eqnarray*}
&&F_k(x,y,\lambda)=\\&&
\frac{\Gamma(2k)\Gamma(3k)}{2^{2k-2}\Gamma(k)^5V(\lambda)^{2k}}\int_{\max(|y|,|x-\lambda_2|)}^{\min(x-\lambda_3, \lambda_1-x)}
\Big( 3z^2(2y+\lambda_1-\lambda_2)-6y(x-\lambda_1)(x-\lambda_2)\Big)\\&&
\qquad\qquad\qquad\qquad\qquad\left(\frac{(\lambda_3-x)^2-z^2}{z^2}\right)^{k}
\Big((z^2-y^2)((\lambda_1-x)^2-z^2)(z^2-(\lambda_2-x)^2)\Big)^{k-1}dz,
\end{eqnarray*}
if $\max(|y|,|x-\lambda_2|)\leq \min(x-\lambda_3, \lambda_1-x)$ and equal $0$ otherwise.
\end{cor}
\begin{proof}

Recall that
\begin{eqnarray*}
\mathcal{J}_{k-\frac{1}{2}}((\mu_1-\mu_2)z)&=&\frac{\Gamma(2k)}{2^{2k-1}\Gamma(k)^2}\int_{\mathbb{R}}
e^{(\mu_1-\mu_2)y}(1-\frac{y^2}{z^2})^{k-1}\chi_{[-1,1]}(\frac{y}{z})z^{-1}dy,\\
\mathcal{J}_{k-\frac{1}{2}}'((\mu_1-\mu_2)z)&=&\frac{\Gamma(2k)}{2^{2k-1}\Gamma(k)^2}\int_{\mathbb{R}}
e^{(\mu_1-\mu_2)y}(1-\frac{y^2}{z^2})^{k-1}\frac{y}{z^2}\chi_{[-1,1]}(\frac{y}{z})dy.
\end{eqnarray*}

Inserting these into  (\ref{IF}) and making use the change of variables: $$x=\frac{\nu_1+\nu_2}{2},\quad z=\frac{\nu_1-\nu_2}{2},$$
with Fubuni's Theorem,  we obtain
\begin{eqnarray}\label{F1}
E_{k}(\mu,\lambda)= \int_{\mathbb{R} }\int_{\mathbb{R} }
 e^{(\mu_1+\mu_2-2\mu_3)x+ (\mu_1-\mu_2)y} F_k(x,y,\lambda)dxdy
\end{eqnarray}
where
\begin{eqnarray*}
&&F_k(x,y,\lambda)=\\&&
\frac{\Gamma(2k)\Gamma(3k)}{2^{2k-2}\Gamma(k)^5V(\lambda)^{2k}}\int_{\mathbb{R} }\Big( 3z^2(\lambda_1-\lambda_2)-6y(x^2-z^2+\lambda_3x+\lambda_1\lambda_2)\Big)
\left(\frac{(\lambda_3-x)^2-z^2}{z^2}\right)^{k}\\&&
\Big((z^2-y^2)(\lambda_1-x)^2-z^2)(z^2-(\lambda_2-x)^2)\Big)^{k-1}\\&&\qquad\qquad\qquad\qquad\qquad\qquad\qquad
  \chi_{[-1,1]} (\frac{y}{z}) \chi_{[\lambda_1,\lambda_2]}(x+z)\chi_{[\lambda_3,\lambda_2]}(x-z)dz\\
=&&\frac{\Gamma(2k)\Gamma(3k)}{2^{2k-2}\Gamma(k)^5V(\lambda)^{2k}}\int_{\max(|y|,|x-\lambda_2|)}^{\min(x-\lambda_3, \lambda_1-x)}
\Big( 3z^2(2y+\lambda_1-\lambda_2)-6y(x-\lambda_1)(x-\lambda_2)\Big)\\&&
\qquad\qquad\qquad\qquad\qquad\left(\frac{(\lambda_3-x)^2-z^2}{z^2}\right)^{k}
\Big((z^2-y^2)((\lambda_1-x)^2-z^2)(z^2-(\lambda_2-x)^2)\Big)^{k-1}dz
\end{eqnarray*}
where we  used the fact that
\begin{eqnarray*}
\chi_{[-1,1]} \left(\frac{y}{z}\right) \chi_{[\lambda_1,\lambda_2]}(x+z)\chi_{[\lambda_3,\lambda_2]}(x-z)=
\chi_{\max(|y|,|x-\lambda_2|)\leq z\leq \min(x-\lambda_3, \lambda_1-x)}.
\end{eqnarray*}
Now,  the change of variables  $$x=\frac{\nu_1+\nu_2}{2},\quad y=\frac{\nu_1-\nu_2}{2},$$
gives
\begin{eqnarray}\label{F1}
E_{k}(\mu,\lambda)= \int_{\mathbb{R} }\int_{\mathbb{R} }
 e^{\langle\mu,\nu \rangle} F_k\left(\frac{\nu_1+\nu_2}{2},\frac{\nu_1-\nu_2}{2},\lambda\right)d\nu_1\nu_2.
\end{eqnarray}
To  achieve the proof we  use that
$$\left\{\nu\in \mathbb{V}; \quad \max \left(\frac{|\nu_1-\nu_2|}{2},\left| \frac{\nu_1+\nu_2}{2}-\lambda_2\right|\right)\leq \min
\Big(\frac{\nu_1+\nu_2}{2}-\lambda_3,\lambda_1-\frac{\nu_1+\nu_2}{2}\Big)\right\}=co(\lambda).$$

\end{proof}

\end{document}